\documentclass[10pt]{article}%

\usepackage[utf8]{inputenc}

\usepackage{amsmath}

\usepackage{amsfonts}
\usepackage{mathrsfs}
\usepackage{amssymb, color}
\usepackage[linkcolor=black,anchorcolor=black,citecolor=black]{hyperref}
\usepackage{graphicx}
\numberwithin{equation}{section}
\usepackage[body={15.5cm,21cm}, top=3cm]{geometry}%
\providecommand{\U}[1]{\protect\rule{.1in}{.1in}}
\providecommand{\U}[1]{\protect \rule{.1in}{.1in}}
\newtheorem{theorem}{Theorem}[section]

\newtheorem{definition}[theorem]{Definition}

\newtheorem{lemma}[theorem]{Lemma}

\newtheorem{proposition}[theorem]{Proposition}
\newtheorem{remark}[theorem]{Remark}

\newtheorem{assumption}[theorem]{Assumption}
\newenvironment{proof}[1][Proof]{\noindent \textbf{#1.} }{\  \rule{0.5em}{0.5em}}

\def \E{\mathsf{E}}

\def \P{\mathsf{P}}

\def \E{\mathsf{E}}
\def \P{\mathsf{P}}

\begin{document}
	\title{Conditional Expectation Backward Stochastic Differential Equations and Related Backward Stochastic Differential Equations with Conditional Reflection}
	\author{ 	Hanwu Li\thanks{Research Center for Mathematics and Interdisciplinary Sciences, Shandong University, Qingdao 266237, Shandong, China. lihanwu@sdu.edu.cn.}
	\thanks{Frontiers Science Center for Nonlinear Expectations (Ministry of Education), Shandong University, Qingdao 266237, Shandong, China.}
    \thanks{Shandong Province Key Laboratory of Financial Risk, Shandong University, Qingdao 266237, Shandong, China.}}
	\date{}
	\maketitle
	
	\begin{abstract}
     In this paper, we introduce a new type of backward stochastic differential equations (BSDEs), called  conditional expectation BSDEs, whose drivers depend not only on the value of the solutions but also on their conditional expectations with respect to a certain sub-$\sigma$-algebra. The collection of these sub-$\sigma$-algebra forms a subfiltration, which stands for partial information that is common for decision making applications. The classical BSDEs and the mean-field BSDEs can be regarded as two special and extreme cases of conditional expectation BSDEs. We  establish the well-posedness for conditional expectation BSDEs under mild conditions and discuss the comparison results. Then, we provide an alternative construction for the solutions to conditional reflected BSDEs without the left-continuity assumption for the subfiltration, which can be seen as the limit of a sequence of penalized conditional expectation BSDEs.
	\end{abstract}


    \textbf{MSC-classification}: 60H10

\section{Introduction}   

In 1990, Pardoux and Peng \cite{PP} first introduced the following nonlinear backward stochastic differential equations (BSDEs)
\begin{align*}
    Y_t=\xi+\int_t^T f(s,Y_s,Z_s)ds-\int_t^T Z_sdB_s.
\end{align*}
For a given square integrable terminal value $\xi$ and a Lipschitz continuous driver $f$, the authors established the existence of a unique pair of solution $(Y,Z)$ to the above equation. It has been widely recognized that the theory of BSDEs provides a useful tool for formulating many problems in mathematical finance, stochastic control, stochastic representation for partial differential equations (PDEs) and differential games (e.g., \cite{DE,EPQ,HL,PP92}). Since then, many extensions have been derived in diverse directions. For example, the regularity of the driver has be weakened in order to deal with the pricing for the European options via exponential utility and the quasilinear PDEs when the nonlinearity has a quadratic growth in the gradient of the solution (see \cite{BH1,BH2,Kobylanski,LS}). Motivated by the models of large stochastic particle systems with mean-field interaction and the mean-field limits in various areas such as statistical mechanics and physics, quantum mechanics and quatum chemistry, Buckdahn et al. \cite{BDLP} investigated a special mean-field problem in a purely stochastic approach and obtained a new kind of BSDEs, called mean-field BSDEs (see \cite{BDL,BLP,BLPR,DEH,L1,L2} and the references therein for more aspects about the mean-field problem). Roughly speaking, the driver of the mean-field BSDE depends on the distribution of the solution. Furthermore, the consideration of additional conditions on the stochastic control problems and mathematical finance naturally led to the investigation of constrained BSDEs (see \cite{BEH,CK,KKPPQ,HHL,HT,Li,PX}). In particular, motivated by the financial problems subject to some state constraints but in the context of partial information, such as pricing American options and recursive reflected utility maximization with partial information, Hu, Huang and Li \cite{HHL} proposed the BSDE with conditional reflection. More precisely, the conditional reflected BSDE associated with the terminal value $\xi$, the driver $f$ and the barrier process $S$ takes the following form
\begin{equation}\label{nonlinearyz}
\begin{cases}
Y_t=\xi+\int_t^T f(s,Y_s,Z_s)ds-\int_t^T Z_s dB_s+K_T-K_t, \\
\E[Y_t|\mathcal{G}_t]\geq \E[S_t|\mathcal{G}_t], \ t\in[0,T], \\
K\in \mathcal{A}^2_{\mathbb{G}} \textrm{ and }
\int_0^T \E[Y_t-S_t|\mathcal{G}_t]dK_t=0.
\end{cases}
\end{equation}
Here, $\mathbb{G}=\{\mathcal{G}_t\}_{t\in[0,T]}$ is a given  subfiltration and $\mathcal{A}^2_{\mathbb{G}}$ is the collection of $\mathbb{G}$-adapted, nondecreasing and continuous function starting from the origin. In \cite{HHL}, the authors established the existence of a solution to the conditional reflected BSDE \eqref{nonlinearyz} by a contraction arguments. 


In this paper, the objective is to investigate the following type of  BSDEs
\begin{equation}\label{CEBSDE}
    Y_t=\xi+\int_t^T f(s,Y_s,Z_s,\E[Y_s|\mathcal{G}_s],\E[Z_s|\mathcal{G}_s])ds-\int_t^T Z_s dB_s.
\end{equation}
As the driver $f$ depends on the conditional expectation of the solution, we call it conditional expectation BSDE.  It is worth pointing out that we do not need to assume that the filtration $\mathbb{G}$ satisfies the ``usual" conditions, i.e., that $\mathcal{G}_0$ contains all $\P$-null set and that $\mathbb{G}$ is right-continuous. Fortunately, we will not encounter any measurability issue under this weak assumption. Indeed,  based on the existence of a progressively measurable modification of a measurable adapted process (see Theorem 1 and Theorem 2 in \cite{KP}), given a pair of measurable processes $(Y,Z)$, the driver $f$ in \eqref{CEBSDE} is progressively measurable. The uniqueness of the solutions to the conditional expectation BSDEs can be obtained by a priori estimates. By constructing a contraction mapping, we establish the existence result. Moreover, the multi-dimensional comparison theorem is obtained under some additional structure of the driver $f$. We should point out that this structure is necessary since we may find some counterexamples in \cite{BLP} for the mean-field BSDEs, which serve as a special kind of conditional expectation BSDEs when the subfiltration is trivial (i.e., $\mathcal{G}_t=\mathcal{F}_0$ for all $t\in[0,T]$). We also obtain the comparison property for the conditional expectation of the solutions. Actually, a similar type of BSDEs \eqref{CEBSDE}, called BSDEs with filtering, has been studied when considering the optimal control problem with partial information (see \cite{LYY,WWY}). However, it should be pointed out that in \cite{LYY,WWY}, the Brownian motion is assumed to be $2$-dimensional, written as $(B^1_t,B^2_t)$, and  the filtration $\mathbb{G}$ is generated by $B^1$. The natural question is whether the BSDE with filtering is well-posed or not when the filtration $\mathbb{G}$ is arbitrarily given. This is one of the reasons why we consider the conditional expectation BSDE.

Clearly, the conditional reflected BSDE \eqref{nonlinearyz} will degenerate into the classical reflected BSDE (resp., the mean reflected BSDE) when the subfiltration $\mathbb{G}$ coincides with the original filtration $\mathbb{F}$ (resp., when the subfiltration $\mathbb{G}$ is trivial and the loss function is linear). For the reflected BSDEs and mean reflected BSDEs, there are two effective methods to construct the solution: the contraction mapping method and the approximation via penalization (see \cite{BEH,CK,KKPPQ,Li}). In this paper, we provide an alternative approach to construct the solution to the conditional reflected BSDEs by the penalization method. This is another motivation to study the conditional expectation BSDEs. The construction is of independent interest and may be useful for the numerical approximation for  conditional reflected BSDEs. Moreover, compared to \cite{HHL}, the advantage is that we do not need to assume that subfiltration $\mathbb{G}$ is left-quasi-continuous (see Remark \ref{remark3.6} below).  

The paper is organized as follows. In Section 2, we establish the well-posedness as well as the comparison theorem for the conditional expectation BSDEs. Then, utilizing the penalization method, we construct the solution to the conditional reflected BSDEs in Section 3.

\section{Conditional expectation BSDEs}


Fix a finite time horizon $T>0$. Let $(\Omega,\mathcal{F},\mathbb{F},\P)$ be a filtered probability space such that $\mathbb{F}=\{\mathcal{F}_t\}_{t\in[0,T]}$ satisfies the usual conditions of right continuity and completeness. Let $B$ be a standard $d$-dimensional Brownian motion and let $\mathbb{G}=\{\mathcal{G}_t\}_{t\in[0,T]}$ be a given  subfiltration, that is, $\mathcal{G}_t\subset \mathcal{F}_t$ for any $t\in[0,T]$ and $\mathcal{G}_s\subset\mathcal{G}_t$ for any $0\leq s\leq t\leq T$. We first introduce the following spaces, which will be frequently used in the sequel. 

\begin{itemize}
\item $L^2(\mathcal{F}_t;\mathbb{R}^n)$: the set of $\mathbb{R}^n$-valued $\mathcal{F}_t$-measurable random variable $\xi$ such that $\E[|\xi|^2]<\infty$.
\item $\mathcal{S}^2_n$: the set of $\mathbb{R}^n$-valued adapted continuous processes $Y$ on $[0,T]$ such that $$\E\left[\sup_{t\in[0,T]}|Y_t|^2\right]<\infty.$$
\item $\mathcal{H}^2_{nd}$: the set of $\mathbb{R}^{n\times d}$-valued predictable processes $Z$ such that $$\E\left[\int_0^T|Z_t|^2dt\right]<\infty.$$
\item If $n=1$, $L^2(\mathcal{F}_t;\mathbb{R}^n)$, $\mathcal{S}^2_n$ and $\mathcal{H}^2_{nd}$ will be denoted by $L^2(\mathcal{F}_t)$,  $\mathcal{S}^2$ and $\mathcal{H}^2_d$, respectively.
\end{itemize}

This section is devoted to the study of the conditional expectation BSDEs \eqref{CEBSDE}. 
In order to obtain the well-posedness of the conditional exepctation BSDEs, we need to propose the following conditions for the driver $f$. 

\begin{assumption}\label{assf}
    The driver $f$ is a map from $\Omega\times[0,T]\times \mathbb{R}^n\times\mathbb{R}^{n\times d}\times\mathbb{R}^n\times\mathbb{R}^{n\times d}$ to $\mathbb{R}^n$. For 
each fixed $(y,z,y',z')$, $f(\cdot,\cdot,y,z,y',z')$ is $\mathbb{F}$-progressively measurable. There exists a constant $\lambda>0$ such that for any $t\in[0,T]$ and any $y_i,y'_i\in\mathbb{R}^n$, $z_i,z'_i\in\mathbb{R}^{n\times d}$, $i=1,2$
\begin{align*}
|f(t,y_1,z_1,y'_1,z'_1)-f(t,y_2,z_2,y'_2,z'_2)|\leq \lambda(|y_1-y_2|+|z_1-z_2|+|y'_1-y'_2|+|z'_1-z'_2|)
\end{align*}
and 
\begin{align*}
\E\left[\int_0^T |f(t,0,0,0,0)|^2dt\right]<\infty.
\end{align*}
\end{assumption}

\begin{remark}
    Suppose that $\mathbb{G}$ degenerates into the deterministic scenario, i.e., $\mathcal{G}_t=\mathcal{F}_0$ for any $t\in[0,T]$. Then the conditional expectation BSDE \eqref{CEBSDE} turns into the mean-field BSDE studied in \cite{BLP}. If $\mathbb{G}=\mathbb{F}$, then the conditional expectation BSDE corresponds to the classical BSDE (see \cite{PP}).
\end{remark}

Throughout this paper,  the conditional expectation of a multi-dimensional random variable is taken under each component. A  solution to the conditional expectation BSDE with terminal value $\xi$ and driver $f$ is a pair of processes $(Y,Z)\in \mathcal{S}^2_n\times\mathcal{H}^2_{nd}$ satisfying \eqref{CEBSDE}. Before investigating the existence and uniqueness of the solution, we need to make sure that the stochastic integral 
\begin{align*}
    \int_0^t f(s,Y_s,Z_s,\E[Y_s|\mathcal{G}_s],\E[Z_s|\mathcal{G}_s])ds
\end{align*}
is well-defined. 
To this end, we need the following definitions and propositions concerning some measurability issue for stochastic processes. Roughly speaking, a measurable adapted process has a progressively measurable modification and can be approximated by a sequence of continuous measurable processes.

\begin{definition}
    An $\bar{\mathbb{R}}$-valued process $X$ is called measurable if the mapping $X:[0,T]\times\Omega\rightarrow\bar{\mathbb{R}}$ is measurable. It is called $\P$-a.s. uniformly sample continuous if there exists a $\P$-null set $N\in\mathcal{F}$ such that for any $\omega\in N^c$, the mapping $t\rightarrow X_t(\omega)$ is uniformly continuous.
\end{definition}

\begin{proposition}[\cite{KP}]\label{measurability}
    Let $X$ be an $\bar{\mathbb{R}}$-valued measurable process adapted to a filtration $\mathbb{G}$. Then it has a $\mathbb{G}$-progressively measurable modification. Moreover, suppose that 
    \begin{align*}
        \E\left[\int_0^T|X_t|^2dt\right]<\infty.
    \end{align*}
    Then, there exists a sequence $\{X^n\}_{n\in\mathbb{N}}$ of bounded $\mathbb{G}$-adapted processes which are $\P$-a.s. uniformly sample continuous, such that 
    \begin{align*}
        \lim_{n\rightarrow \infty} \E\left[\int_0^T|X_t-X^n_t|^2dt\right]=0.
    \end{align*}
\end{proposition}

\begin{remark}\label{r1}
  (i) If the driver $f$ does not depend on the last two component $(y',z')$ and satisfies Assumption \ref{assf}, we call $f$ satisfies the usual conditions.
  
 \noindent (ii) Given $(U,V)\in \mathcal{S}^2_n\times\mathcal{H}^2_{nd}$, by Assumption \ref{assf} and Proposition \ref{measurability}, the following  integral 
    \begin{align*}
    \int_0^t f(s,U_s,V_s,\E[U_s|\mathcal{G}_s],\E[V_s|\mathcal{G}_s])ds
\end{align*}
is well-defined.  Now, we define
\begin{align*}
        f^{U,V}(s,y,z):=f(s,y,z,\E[U_s|\mathcal{G}_s],\E[V_s|\mathcal{G}_s]).
    \end{align*}
By Assumption \ref{assf} and Proposition \ref{measurability}, it is easy to check that $f$ satisfies the usual conditions. 
\end{remark}

We first provide some a priori estimates for the solutions to the  conditional expectation BSDEs.

\begin{proposition}\label{aprioriestimate}
    Given $\xi^i\in L^2(\mathcal{F}_T;\mathbb{R}^n)$ and $f^i$ satisfying Assumption \ref{assf}, let $(Y^i,Z^i)$ be the solution to the conditional expectation BSDE with terminal value $\xi^i$ and driver $f^i$, $i=1,2$. Set $\hat{Y}_t=Y^1_t-Y^2_t$, $\hat{Z}_t=Z^1_t-Z^2_t$ and 
    \begin{align*}
        \hat{f}^2_s=f^1\left(s,Y^2_s,Z^2_s,\E\left[Y^2_s\big|\mathcal{G}_s\right],\E\left[Z^2_s\big|\mathcal{G}_s\right]\right)-f^2\left(s,Y^2_s,Z^2_s,\E\left[Y^2_s\big|\mathcal{G}_s\right],\E\left[Z^2_s\big|\mathcal{G}_s\right]\right).
    \end{align*}
    Then, there exists a constant $C$ depending on $\lambda, T$, such that 
    \begin{align*}
        \E\left[\sup_{t\in[0,T]}|\hat{Y}_t|^2\right]+\E\left[\int_0^T |\hat{Z}_s|^2 ds\right]\leq C\left(\E\left[|\hat{\xi}|^2\right]+\E\left[\int_0^T |\hat{f}^2_s|^2 ds\right]\right).
    \end{align*}
\end{proposition}

\begin{proof}
    Set
    \begin{align*}
        \hat{f}_s=f^1\left(s,Y^1_s,Z^1_s,\E\left[Y^1_s\big|\mathcal{G}_s\right],\E\left[Z^1_s\big|\mathcal{G}_s\right]\right)-f^2\left(s,Y^2_s,Z^2_s,\E\left[Y^2_s\big|\mathcal{G}_s\right],\E\left[Z^2_s\big|\mathcal{G}_s\right]\right).
    \end{align*}
    Applying It\^{o}'s formula to $e^{\beta t}|\hat{Y}|^2_t$ and taking conditional expectations on both sides, we obtain that 
    \begin{align*}
        &\E\left[e^{\beta t}|\hat{Y}_t|^2\right]+\beta\E\left[\int_t^T e^{\beta s}|\hat{Y}_s|^2 ds\right]+\E\left[\int_t^T e^{\beta s}|\hat{Z}_s|^2 ds\right]
        = \E\left[e^{\beta T}|\hat{\xi}|^2\right]+\E\left[\int_t^T e^{\beta s}2\hat{Y}_s\hat{f}_sds\right]\\
        \leq &\E\left[e^{\beta T}|\hat{\xi}|^2\right]+\E\left[\int_t^T e^{\beta s}2|\hat{Y}_s|\left(|\hat{f}^2_s|+\lambda\left(|\hat{Y}_s|+|\hat{Z}_s|+\E\left[|\hat{Y}_s|\big|\mathcal{G}_s\right]+\E\left[|\hat{Z}_s|\big|\mathcal{G}_s\right]\right)\right)ds\right]\\
        \leq &\E\left[e^{\beta T}|\hat{\xi}|^2\right]+\E\left[\int_t^T e^{\beta s}|\hat{f}^2_s|^2 ds\right]+(2\lambda+4\lambda^2+1)\E\left[\int_t^T e^{\beta s}|\hat{Y}_s|^2 ds\right]+\frac{1}{4}\E\left[\int_t^T e^{\beta s}|\hat{Z}_s|^2 ds\right]\\
        &+\E\left[\int_t^T e^{\beta s}2\lambda|\hat{Y}_s|\E\left[|\hat{Y}_s|\big|\mathcal{G}_s\right] ds\right]+\E\left[\int_t^T e^{\beta s}2\lambda|\hat{Y}_s|\E\left[|\hat{Z}_s|\big|\mathcal{G}_s\right] ds\right].
    \end{align*}
    It is easy to check that 
    \begin{align*}
        \E\left[\int_t^T e^{\beta s}2\lambda|\hat{Y}_s|\E\left[|\hat{Y}_s|\big|\mathcal{G}_s\right] ds\right]=\int_t^T e^{\beta s}2\lambda\E\left[\left(\E\left[|\hat{Y}_s|\big|\mathcal{G}_s\right]\right)^2\right] ds\leq 2\lambda \E\left[\int_t^T e^{\beta s}|\hat{Y}_s|^2 ds\right]
    \end{align*}
    and 
    \begin{align*}
       &\E\left[\int_t^T e^{\beta s}2\lambda|\hat{Y}_s|\E\left[|\hat{Z}_s|\big|\mathcal{G}_s\right] ds\right]\\
       =& \int_t^T e^{\beta s}2\lambda\E\left[\E\left[|\hat{Y}_s|\big|\mathcal{G}_s\right]\E\left[|\hat{Z}_s|\big|\mathcal{G}_s\right]\right] ds\\
       \leq &\int_t^T e^{\beta s}2\lambda \left(\E\left[\left(\E\left[|\hat{Y}_s|\big|\mathcal{G}_s\right]\right)^2\right]\right)^{1/2}\left(\E\left[\left(\E\left[|\hat{Z}_s|\big|\mathcal{G}_s\right]\right)^2\right]\right)^{1/2}ds\\
       \leq &4\lambda^2\E\left[\int_t^T e^{\beta s}|\hat{Y}_s|^2 ds\right]+\frac{1}{4}\E\left[\int_t^T e^{\beta s}|\hat{Z}_s|^2 ds\right].
    \end{align*}
    Then, setting $\beta=4\lambda+8\lambda^2+2$, all the above analysis indicates that there exists a constant $C$ depending on $\lambda,T$, such that for any $t\in[0,T]$, we have
    \begin{align*}
        \E\left[|\hat{Y}_t|^2\right]+\E\left[\int_t^T |\hat{Y}_s|^2 ds\right]+\E\left[\int_t^T |\hat{Z}_s|^2 ds\right]\leq C\left(\E\left[|\hat{\xi}|^2\right]+\E\left[\int_t^T |\hat{f}^2_s|^2 ds\right]\right).
    \end{align*}
    Applying the above estimate, the H\"older inequality and Burkholder-Davis-Gundy's inequality, we finally obtain the estimate for $\E[\sup_{t\in[0,T]}|\hat{Y}_t|^2]$. The proof is complete.
\end{proof}

Now, we are ready to present the main result in this section.

\begin{theorem}\label{wellposedness}
    Given $\xi\in L^2(\mathcal{F}_T;\mathbb{R}^n)$, suppose that the driver $f$ satisfies Assumption \ref{assf}. Then, the conditional expectation BSDE \eqref{CEBSDE} admits a unique solution $(Y,Z)\in \mathcal{S}^2_n\times\mathcal{H}^2_{nd}$. 
\end{theorem}

\begin{proof}
Uniqueness is a direct consequence of Proposition \ref{aprioriestimate}. We are in the position to show the existence. In this proof, we use the notations in Remark \ref{r1}.  Given $(U,V)\in \mathcal{S}^2_n\times\mathcal{H}^2_{nd}$,  by Remark \ref{r1} and the result in \cite{PP}, the following BSDE
    \begin{align*}
        Y^{U,V}_t=\xi+\int_t^T f(s,Y^{U,V}_s,Z^{U,V}_s,\E[U_s|\mathcal{G}_s],\E[V_s|\mathcal{G}_s])ds-\int_t^T Z^{U,V}_s dB_s
    \end{align*}
 admits a unique solution $(Y^{U,V},Z^{U,V})\in\mathcal{S}^2_n\times \mathcal{H}^2_{nd}$. We define the mapping $I:\mathcal{H}^2_n\times\mathcal{H}^2_{nd}\rightarrow\mathcal{H}^2_{n}\times\mathcal{H}^2_{nd}$ as follows
    \begin{align*}
        I(U,V)=(Y^{U,V},Z^{U,V}).
    \end{align*}
    Now, for any given $(U^i,V^i)\in \mathcal{H}^2_n\times\mathcal{H}^2_{nd}$, we write $(Y^i,Z^i)=I(U^i,Z^i)$, $i=1,2$. Set $\hat{X}_t=X^1_t-X^2_t$ for $X=U,V,Y,Z$. Applying It\^{o}'s formula to $e^{\beta t}|\hat{Y}|^2_t$ and taking conditional expectations w.r.t. $\mathcal{F}_t$, we have
    \begin{align*}
        &e^{\beta t}|\hat{Y}_t|^2+\beta\E\left[\int_t^T e^{\beta s}|\hat{Y}_s|^2 ds\Big|\mathcal{F}_t\right]+\E\left[\int_t^T e^{\beta s}|\hat{Z}_s|^2 ds\Big|\mathcal{F}_t\right]\\
        = &\E\left[\int_t^T e^{\beta s}2\hat{Y}_s\left(f^{U^1,V^1}(s,Y^1_s,Z^1_s)-f^{U^2,V^2}(s,Y^2_s,Z^2_s)\right)ds\Big|\mathcal{F}_t\right].
    \end{align*}
    By Assumption \ref{assf}, it is easy to check that 
    \begin{align*}
        &\int_t^T e^{\beta s}2\hat{Y}_s\left(f^{U^1,V^1}(s,Y^1_s,Z^1_s)-f^{U^2,V^2}(s,Y^2_s,Z^2_s)\right)ds\\
        \leq &(2\lambda+2\lambda^2+\frac{\beta}{2})\int_t^T e^{\beta s}|\hat{Y}_s|^2ds+\frac{1}{2}\int_t^T e^{\beta s}|\hat{Z}_s|^2ds\\
        &+\frac{4\lambda^2}{\beta}\left(\int_t^T e^{\beta s}\E\left[|\hat{U}_s|^2\big|\mathcal{G}_s\right]ds+\int_t^T e^{\beta s}\E\left[|\hat{V}_s|^2\big|\mathcal{G}_s\right]ds\right).
    \end{align*}
    Combining the above two equations and taking expectations yield that for any $t\in[0,T]$
    \begin{align*}
        &\left(\frac{\beta}{2}-2\lambda-2\lambda^2\right)\E\left[\int_t^T e^{\beta s}|\hat{Y}_s|^2 ds\right]+\frac{1}{2}\E\left[\int_t^T e^{\beta s}|\hat{Z}_s|^2 ds\right]\\
        \leq &\frac{4\lambda^2}{\beta}\left(\E\left[\int_t^T e^{\beta s}|\hat{U}_s|^2ds\right]+\E\left[\int_t^T e^{\beta s}|\hat{V}_s|^2ds\right]\right).
    \end{align*}
    Choosing $\beta=16\lambda^2+4\lambda+1$ and $t=0$, we construct a contraction mapping $I$ on $\mathcal{H}^2_n\times\mathcal{H}^2_{nd}$ endowed with the norm $\|\cdot\|_\beta$, where 
    \begin{align*}
        \|(U,V)\|_\beta^2=\E\left[\int_0^T e^{\beta s}|\hat{U}_s|^2ds\right]+\E\left[\int_0^T e^{\beta s}|\hat{V}_s|^2ds\right].
    \end{align*}
    Therefore, there exists a unique fixed point $(Y,Z)\in \mathcal{H}^2_n\times\mathcal{H}^2_{nd}$ such that $I(Y,Z)=(Y,Z)$. Moreover, by H\"{o}lder's inequality and  Burkholder-Davis-Gundy's inequality, the first component of $I(Y,Z)$ belongs to $\mathcal{S}^2_n$. The proof is complete.
\end{proof}

\begin{remark}
    Suppose that $B=(B^1,B^2)$ is a $2$-dimensional Brownian motion and the subfiltration is generated by $B^1$, i.e., for any $t\in[0,T]$, $\mathcal{G}_t=\mathcal{F}^{B^1}_t$. Then, Theorem \ref{wellposedness} degenerates into Lemma 4.1 in \cite{WWY}.
\end{remark}

One of the most important results in the theory of BSDE is the comparison theorem. Recall that Example 3.1 and Example 3.2 in \cite{BLP} indicates that the comparison theorem for mean-field BSDEs may not hold when  both drivers depend on the expectation of $Z$-term or are decreasing with respect to the expectation of $Y$-term. In order to derive the comparison theorem in the present framework, some additional structure of the driver should be imposed. Motivated by \cite{LXP}, we propose the following assumptions. 

\begin{itemize}
    \item[(H1)] For each fixed $j=1,2,\cdots,n$, and for all $(\omega,t,y',z')\in\Omega\times[0,T]\times\mathbb{R}^n\times\mathbb{R}^{n\times d}$, $y^1,y^2\in\mathbb{R}^n$, $z^1,z^2\in\mathbb{R}^{n\times d}$ with $y^1_j=y^2_j$, $z^1_j=z^2_j$, $y^1_l\leq y^2_l$, $l\neq j$, we have
    \begin{align*}
        f^1_j(\omega,t,y^1,z^1,y',z')\leq f^2_j(\omega,t,y^2,z^2,y',z').
    \end{align*}
    \item[(H2)] For each fixed $j=1,2,\cdots,n$, for all $y',y''\in\mathbb{R}^n$, and all $(t,y,z,z')\in [0,T]\times\mathbb{R}^n\times\mathbb{R}^{n\times d}\times\mathbb{R}^{n\times d}$, there exists a constant $L>0$, such that 
    \begin{align*}
        f_j(t,y,z,y',z')-f_j(t,y,z,y'',z')\leq L(y'-y'')^+,
    \end{align*}
    where $(y'-y'')^+=(\sum_{j=1}^n|(y'_j-y''_j)^+|^2)^{1/2}$.
\end{itemize}

\begin{theorem}\label{comparison}
    Let $f^i$, $i=1,2$, be two drivers satisfying Assumption \ref{assf} and (H1). Moreover, we assume that 
    \begin{itemize}
        \item[(i)] One of the drivers is independent of $z'$;
        \item[(ii)] One of the drivers satisfies (H2).
    \end{itemize}
    Given $\xi^i\in L^2(\mathcal{F}_T;\mathbb{R}^n)$, $i=1,2$, let $(Y^i,Z^i)$ be the solution to the following conditional expectation BSDE 
    \begin{equation}\begin{split}\label{CEBSDEK}
        Y^i(t)=&\xi^i+\int_t^T f^i(s,Y^i(s),Z^i(s),\E[Y^i(s)|\mathcal{G}_s],\E[Z^i(s)|\mathcal{G}_s])ds\\
        &-\int_t^T Z^i(s) dB_s+K^i(T)-K^i(t),
    \end{split}\end{equation}
    where $K^i\in \mathcal{S}^2_n$, $i=1,2$. For any $j=1,2,\cdots,n$, suppose that $\xi^1_j\leq \xi^2_j$, $\{K^2_j(t)-K^1_j(t)\}_{t\in[0,T]}$ is nondecreasing. Then, for any $j=1,2,\cdots,n$, we have $Y^1_j(t)\leq Y^2_j(t)$, for all $t\in[0,T]$, $\P$-a.s.
\end{theorem}

\begin{proof}
    Without loss of generality, we assume that $f^1$ is independent of $z'$ and satisfies (H2). Set 
    $$\widehat{Y}(t)=Y^1(t)-Y^2(t),\quad\quad\widehat{Z}_t=Z^1(t)-Z^2(t),\quad\quad\widehat{K}(t)=K^1(t)-K^2(t)$$
    and $$\widehat{f}(t)=f^1(t,Y^1(t),Z^1(t),\E[Y^1(t)|\mathcal{G}_t])-f^2(t,Y^2(t),Z^2(t),\E[Y^2(t)|\mathcal{G}_t],\E[Z^2(t)|\mathcal{G}_t]).$$ 
    By the Tanaka-Meyer formula, for any $j=1,2,\cdots,n$, we have
    \begin{align*}
        d\widehat{Y}^+_j(t)=I_{\{\widehat{Y}_j(t)>0\}}\widehat{Z}_j(t)dB_t-I_{\{\widehat{Y}_j(t)>0\}}\widehat{f}_j(t)dt-I_{\{\widehat{Y}_j(t)>0\}}d\widehat{K}_j(t)+dC_j(t),
    \end{align*}
    where $\{C_j(t)\}_{t\in[0,T]}$ is an adpated, continuous, and nondecreasing process related with the local time for $\widehat{Y}_j$.  Applying It\^{o}'s formula to $(\widehat{Y}_j(t)^+)^2$ and taking expectations on both sides,  for any $t\in[0,T]$, we have 
    \begin{align*}
        &\E[(\widehat{Y}^+_j(t))^2]+\E\left[\int_t^T |\widehat{Z}_j(s)|^2 I_{\{\widehat{Y}_j(s)>0\}}ds\right]+2\E\left[\int_t^T\widehat{Y}^+_j(s)dC_j(s)\right]\\
        =&2\E\left[\int_t^T\widehat{Y}^+_j(s)\widehat{f}_j(s) ds\right]+2\E\left[\int_t^T\widehat{Y}^+_j(s) d\widehat{K}_j(s)\right]\leq 2\E\left[\int_t^T\widehat{Y}^+_j(s)\widehat{f}_j(s) ds\right].
    \end{align*}
    Set 
    \begin{align*}
        Y^{1,2}_j(s)&:=(Y^1_1(s),\cdots,Y^1_{j-1}(s), Y^2_j(s),Y^1_{j+1}(s),\cdots,Y^1_n(s)),\\
        \widehat{Y}^+_j(s)&:=(\widehat{Y}^+_1(s),\cdots,\widehat{Y}^+_{j-1}(s),0,\widehat{Y}^+_{j+1}(s),\cdots,\widehat{Y}^+_n(s)),\\
        Z^{1,2}_j(s)&:=(Z^1_1(s),\cdots,Z^1_{j-1}(s), Z^2_j(s),Z^1_{j+1}(s),\cdots,Z^1_n(s)).
    \end{align*}
    Since $f^1$ is independent of $z'$ and satisfies (H2), it is easy to check that
    \begin{align*}
        \widehat{f}_j(s) =&f^1_j(s,Y^1(s),Z^1(s),\E[Y^1(s)|\mathcal{G}_s])-f^1_j(s,Y^1(s),Z^1(s),\E[Y^2(s)|\mathcal{G}_s])\\
        &+f^1_j(s,Y^1(s),Z^1(s),\E[Y^2(s)|\mathcal{G}_s])-f^1_j(s,Y^{1,2}_j(s)-\widehat{Y}^+_j(s),Z^{1,2}(s),\E[Y^2(s)|\mathcal{G}_s])\\
        &+f^1_j(s,Y^{1,2}_j(s)-\widehat{Y}^+_j(s),Z^{1,2}(s),\E[Y^2(s)|\mathcal{G}_s])-f^2_j(s,Y^2(s),Z^2(s),\E[Y^2(s)|\mathcal{G}_s],\E[Z^2(s)|\mathcal{G}_s])\\
        \leq & L(\E[Y^1(s)|\mathcal{G}_s]-\E[Y^2(s)|\mathcal{G}_s])^++\lambda|Y^1(s)-(Y^{1,2}_j(s)-\widehat{Y}^+_j(s))|+\lambda|Z^1(s)-Z^{1,2}(s)|\\
        =&L\left(\sum_{i=1}^n \left|(\E[\widehat{Y}_i(s)|\mathcal{G}_s])^+\right|^2\right)^{1/2}+\lambda\left(\sum_{i\neq j}|\widehat{Y}^+_i(s)|^2+|\widehat{Y}_j(s)|^2\right)^{1/2}+\lambda|\widehat{Z}_j(s)|.
    \end{align*}
    Simple calculation yields that 
\begin{align*}
    2\int_t^T\widehat{Y}^+_j(s)\widehat{f}_j(s) ds
    \leq &L\int_t^T\left(\widehat{Y}^+_j(s)\right)^2ds+L\int_t^T\sum_{i=1}^n \E\left[\left(\widehat{Y}^+_i(s)\right)^2\Big|\mathcal{G}_s\right]ds\\
    &+2\lambda\int_t^T \widehat{Y}^+_j(s)\left(\sum_{i\neq j}|\widehat{Y}^+_i(s)|+|\widehat{Y}_j(s)|\right)ds\\
    &+2\lambda^2\int_t^T\left(\widehat{Y}^+_j(s)\right)^2ds+\frac{1}{2}\int_t^T |\widehat{Z}_j(s)|^2 I_{\{\widehat{Y}_j(s)>0\}}ds\\
    \leq &(L+2\lambda^2)\int_t^T\left(\widehat{Y}^+_j(s)\right)^2ds+L\int_t^T\sum_{i=1}^n \E\left[\left(\widehat{Y}^+_i(s)\right)^2\Big|\mathcal{G}_s\right]ds\\
    &+2\lambda\int_t^T\left(\widehat{Y}^+_j(s)\right)^2ds+\lambda(n-1)\int_t^T\left(\widehat{Y}^+_j(s)\right)^2ds\\
    &+\lambda\sum_{i\neq j}\int_t^T\left(\widehat{Y}^+_i(s)\right)^2ds+\frac{1}{2}\int_t^T |\widehat{Z}_j(s)|^2 I_{\{\widehat{Y}_j(s)>0\}}ds.
\end{align*}
      All the above analysis indicates that 
    \begin{align*}
        \E\left[\left(\widehat{Y}^+_j(t)\right)^2\right]
        \leq& (L+2\lambda^2+\lambda n+\lambda)\int_t^T\E\left[\left(\widehat{Y}^+_j(s)\right)^2\right]ds\\
        &+L\int_t^T\sum_{i=1}^n \E\left[\left(\widehat{Y}^+_i(s)\right)^2\right]ds+\lambda\int_t^T\sum_{i\neq j} \E\left[\left(\widehat{Y}^+_i(s)\right)^2\right]ds.
    \end{align*}
    Taking the sum from $1$ to $n$ yields that, there exists a constant $M$ depending on $\lambda,n,L$, such that 
    \begin{align*}
        \E\left[\sum_{j=1}^n\left(\widehat{Y}^+_j(t)\right)^2\right]\leq M\int_t^T\E\left[\sum_{j=1}^n\left(\widehat{Y}^+_j(s)\right)^2\right]ds.
    \end{align*}
    Applying Gr\"onwall's inequality, we obtain that  $\E[\sum_{j=1}^n(\widehat{Y}^+_j(t))^2]=0$ for any $t\in[0,T]$. The proof is complete.
\end{proof}

\begin{remark}
   (i) Suppose that $n=1$. (H1) is the common assumption made for the comparison theorem, i.e., for any $(\omega,t,y,z,y',z')\in\Omega\times[0,T]\times\mathbb{R}\times\mathbb{R}^d\times\mathbb{R}\times\mathbb{R}^d$, we have
   \begin{align*}
       f^1(\omega,t,y,z,y',z')\leq f^2(\omega,t,y,z,y',z').
   \end{align*}
   (H2) clearly holds when $f$ is nondecreasing in $y'$. 
    
    \noindent (ii) Suppose $f$ satisfies Assumption \ref{assf} and is increasing in component $y'$ in the following sense:
    \begin{itemize}
        \item For each fixed $j=1,2,\cdots,n$, for all $(t,y,z,z')\in [0,T]\times\mathbb{R}^n\times\mathbb{R}^{n\times d}\times\mathbb{R}^{n\times d}$, for all $i=1,2,\cdots,n$ and $y',y''\in\mathbb{R}^n$ with $y'_l=y''_l$, $l\neq i$ and $y'_i\leq y''_i$, we have
        \begin{align*}
        f_j(t,y,z,y',z')-f_j(t,y,z,y'',z')\leq 0.
    \end{align*}
    \end{itemize}
    Then, $f$ satisfies (H2). We only show the case that $n=2$. In fact, simple calculation implies that 
    \begin{align*}
        &f_j(t,y,z,y',z')-f_j(t,y,z,y'',z')\\
        =&f_j(t,y,z,y'_1,y'_2,z')-f_j(t,y,z,y''_1,y'_2,z')\\
        &+f_j(t,y,z,y''_1,y'_2,z')-f_j(t,y,z,y''_1,y''_2,z')\\
        \leq &\lambda(y'_1-y''_1)^++\lambda(y'_2-y''_2)^+\leq 2^{1/2}\lambda(y'-y'')^+.
    \end{align*}
\end{remark}

We also have a converse comparison theorem.

\begin{theorem}\label{conversecomparison}
  Under the assumptions of Theorem \ref{comparison}, suppose that for some $t\in[0,T]$, $Y^1_j(t)=Y^2_j(t)$, for all $j=1,2,\cdots,n$. Then, for any $s\in[t,T]$ and $j=1,2,\cdots,n$, we have $Y^1_j(s)=Y^2_j(s)$.  
\end{theorem}

\begin{proof}
    We use the same notation introduced in the proof of Theorem \ref{comparison}. Suppose that $f^1$ is independent of $z'$ and satisfies (H2). For any $j=1,2,\cdots,n$, there exist two bounded, measurable processes $\alpha_j,\beta_j$, such that 
    \begin{align*}
        f^1_j(s,Y^1(s),Z^1(s),\E[Y^1(s)|\mathcal{G}_s])-f^1_j(s,Y^{1,2}_j(s),Z^{1,2}_j(s),\E[Y^1(s)|\mathcal{G}_s])=a_j(s)\widehat{Y}_j(s)+\beta_j(s)\widehat{Z}_j(s).
    \end{align*}
   Applying It\^{o}'s formula to $\exp(\int_0^t \alpha_j(r)dr)\widehat{Y}_j(t)$, we obtain that for any $u\in[t,T]$,
\begin{equation}\begin{split}\label{widehatyj}
    \widehat{Y}_j(t)=&\exp\left(\int_t^u \alpha_j(r)dr\right)\widehat{Y}_j(u)+\int_t^u \exp\left(\int_t^s \alpha_j(r)dr\right)\widetilde{f}_j(s) ds\\
    &-\int_t^u \exp\left(\int_t^s \alpha_j(r)dr\right)\widehat{Z}_j(s)d\widetilde{B}^j_s-\int_t^u \exp\left(\int_t^s \alpha_j(r)dr\right)d\widehat{K}_j(s),
\end{split}\end{equation}
 where $\widetilde{B}^j_s=B_s-\int_0^s \beta_j(r)dr$ and 
 \begin{align*}
     \widetilde{f}_j(s)=f^1_j(s,Y^{1,2}_j(s),Z^{1,2}_j(s),\E[Y^1(s)|\mathcal{G}_s])-f^2_j(s,Y^2(s),Z^2(s),\E[Y^2(s)|\mathcal{G}_s],\E[Z^2(s)|\mathcal{G}_s]).
 \end{align*}
 Let $\P^j$ be the probability measure such that 
 \begin{align*}
     \frac{d\P^j}{d\P}=\exp\left(\int_0^T \beta_j(s)dB_s-\frac{1}{2}\int_0^T|\beta_j(s)|^2ds\right).
 \end{align*}
 Noting that $\widehat{K}_j$ is nondecreasing and taking conditional expectations on both sides of \eqref{widehatyj} under $\P^j$ with respect to $\mathcal{F}_t$ imply that 
 \begin{align*}
     0=\widehat{Y}_j(t)\leq\E^{\P^j}\left[\exp\left(\int_t^u \alpha_j(r)dr\right)\widehat{Y}_j(u)+\int_t^u \exp\left(\int_t^s \alpha_j(r)dr\right)\widetilde{f}_j(s) ds\Big|\mathcal{F}_t\right].
 \end{align*}
 Recalling Theorem \ref{comparison} and assumptions (H1), (H2), for any $s\in[0,T]$, we have $\widehat{Y}_j(s)\leq 0$ and 
 \begin{align*}
     \widetilde{f}_j(s)=&f^1_j(s,Y^{1,2}_j(s),Z^{1,2}_j(s),\E[Y^1(s)|\mathcal{G}_s])-f^1_j(s,Y^{1,2}_j(s),Z^{1,2}_j(s),\E[Y^2(s)|\mathcal{G}_s])\\
     &+f^1_j(s,Y^{1,2}_j(s),Z^{1,2}_j(s),\E[Y^2(s)|\mathcal{G}_s])-f^2_j(s,Y^2(s),Z^2(s),\E[Y^2(s)|\mathcal{G}_s],\E[Z^2(s)|\mathcal{G}_s])\leq 0.
 \end{align*}
 We finally obtain the desired result.
\end{proof}

Motivated by the comparison theorem for conditional reflected BSDEs (see Corollary 3.3 in \cite{HHL}), in the following, we compare the conditional expectation of the solutions to conditional expectation BSDEs when $n=1$.
\begin{proposition}\label{comparison'}
 Let $f_i$, $i=1,2$, be two drivers satisfying Assumption \ref{assf}. Let $f_2$ be independent of $(z,z')$. 
   Given $\xi^i\in L^2(\mathcal{F}_T)$ and $K^i\in \mathcal{S}^2$, $i=1,2$, let $(Y^i,Z^i)$ be the solution to the conditional expectation BSDE \eqref{CEBSDEK}. Suppose that
    \begin{itemize}
        \item[(1)] $\E[\xi^1|\mathcal{G}_T]\leq \E[\xi^2|\mathcal{G}_T]$,
        \item[(2)] $\{\E[K^2_t-K^1_t|\mathcal{G}_t]\}_{t\in[0,T]}$ (or $\{K^2_t-K^1_t\}_{t\in[0,T]}$) is a nondecreasing process,
        \item[(3)] for any $t\in[0,T]$,
    \begin{align}\label{eqf1f2}
        \E\left[f_1\left(t,Y^1_t,Z^1_t,\E\left[Y^1_t\big|\mathcal{G}_t\right],\E\left[Z^1_t\big|\mathcal{G}_t\right]\right)|\mathcal{G}_t\right]\leq \E\left[f_2\left(t,Y^1_t,\E\left[Y^1_t\big|\mathcal{G}_t\right]\right)|\mathcal{G}_t\right],
    \end{align}
    \item[(4)] the process $\{a_t\}_{t\in[0,T]}$ is $\mathbb{G}$-adapted and for any $t\in[0,T]$, we have $a_t+b_t\geq 0$, where 
    \begin{align*}
    a_t&=\frac{f_2\left(t,Y^1_t,\E[Y^1_t|\mathcal{G}_t]\right)-f_2\left(t,Y^2_t,\E[Y^1_t|\mathcal{G}_t]\right)}{Y^1_t-Y^2_t}I_{\{Y^1_t\neq Y^2_t\}},\\
        b_t&=\frac{f_2\left(t,Y^2_t,\E[Y^1_t|\mathcal{G}_t]\right)-f_2\left(t,Y^2_t,\E[Y^2_t|\mathcal{G}_t]\right)}{\E[Y^1_t|\mathcal{G}_t]-\E[Y^2_t|\mathcal{G}_t]}I_{\{\E[Y^1_t|\mathcal{G}_t]\neq\E[Y^2_t|\mathcal{G}_t]\}}.
    \end{align*}
    \end{itemize}
    Then, we have $\E[Y^1_t|\mathcal{G}_t]\leq \E[Y^2_t|\mathcal{G}_t]$, for all $t\in[0,T]$, $\P$-a.s.
\end{proposition}

\begin{proof}
     Using the same notations as in the proof of Theorem \ref{comparison} and letting $\widehat{\xi}=\xi^1-\xi^2$, it is easy to check that 
    \begin{align*}
        \widehat{Y}_t=\widehat{\xi}+\int_t^T \widehat{f}_sds-\int_t^T\widehat{Z}_sdB_s+\widehat{K}_T-\widehat{K}_t.
    \end{align*}
    Taking conditional expectations w.r.t. $\mathcal{G}_t$ on both sides yields that 
    \begin{align*}
        \E\left[\widehat{Y}_t\big|\mathcal{G}_t\right]=&\E\left[\widehat{\xi}\big|\mathcal{G}_t\right]+\int_t^T \E\left[\widehat{f}_s\big|\mathcal{G}_t\right]ds+\E\left[\widehat{K}_T-\widehat{K}_t\big|\mathcal{G}_t\right]\\
        =&\E\left[\E\left[\widehat{\xi}\big|\mathcal{G}_T\right]\big|\mathcal{G}_t\right]+\int_t^T \E\left[\widehat{f}_s\big|\mathcal{G}_t\right]ds+\E\left[\E\left[\widehat{K}_T\big|\mathcal{G}_T\right]-\E\left[\widehat{K}_t\big|\mathcal{G}_t\right]\big|\mathcal{G}_t\right]\\
        \leq &\int_t^T \E\left[\widehat{f}_s\big|\mathcal{G}_t\right]ds.
    \end{align*}
    It is easy to check that 
    \begin{align*}
        \E\left[\widehat{f}_s\big|\mathcal{G}_t\right]=&\E\left[\left(\widehat{f}\left(s,Y^1_s,Z^1_s,\E\left[Y^1_s\big|\mathcal{G}_s\right],\E\left[Z^1_s\big|\mathcal{G}_s\right]\right)+a_s\widehat{Y}_s+b_s\E\left[\widehat{Y}_s\big|\mathcal{G}_s\right]\right)\Big|\mathcal{G}_t\right]\\
        \leq &\E\left[(a_s+b_s)\E\left[\widehat{Y}_s\big|\mathcal{G}_s\right]\Big|\mathcal{G}_t\right]\leq 2\lambda\E\left[\left(\E\left[\widehat{Y}_s\big|\mathcal{G}_s\right]\right)^+\Big|\mathcal{G}_t\right],
    \end{align*}
    where
    \begin{align*}
        \widehat{f}\left(s,Y^1_s,Z^1_s,\E\left[Y^1_s\big|\mathcal{G}_s\right],\E\left[Z^1_s\big|\mathcal{G}_s\right]\right)=f_1\left(s,Y^1_s,Z^1_s,\E\left[Y^1_s\big|\mathcal{G}_s\right],\E\left[Z^1_s\big|\mathcal{G}_s\right]\right)-f_2\left(s,Y^1_s,\E\left[Y^1_s\big|\mathcal{G}_s\right]\right).
    \end{align*}
    All the above analysis indicates that 
    \begin{align*}
        \left(\E\left[\widehat{Y}_t\big|\mathcal{G}_t\right]\right)^+\leq 2\lambda\int_t^T\E\left[\left(\E\left[\widehat{Y}_s\big|\mathcal{G}_s\right]\right)^+\Big|\mathcal{G}_t\right]ds.
    \end{align*}
    Consequently, for any $t\in[0,T]$, we have 
    \begin{align*}
        \E\left[\left(\E\left[\widehat{Y}_t\big|\mathcal{G}_t\right]\right)^+\right]\leq 2\lambda\int_t^T\E\left[\left(\E\left[\widehat{Y}_s\big|\mathcal{G}_s\right]\right)^+\right]ds.
    \end{align*}
    It follows from the Gr\"onwall inequality that $\E[(\E[\widehat{Y}_t\big|\mathcal{G}_t])^+]=0$ for any $t\in[0,T]$, which implies the desired result. 
\end{proof}

\begin{remark}
(i)    Suppose that $f_1$ is independent of $(z,z')$ and satisfies condition (4) in Proposition \ref{comparison'}. 
The result in Proposition \ref{comparison'} still holds if \eqref{eqf1f2} is replaced by 
    \begin{align*}
        \E\left[f_2\left(t,Y^2_t,Z^2_t,\E\left[Y^2_t\big|\mathcal{G}_t\right],\E\left[Z^2_t\big|\mathcal{G}_t\right]\right)|\mathcal{G}_t\right]\geq \E\left[f_1\left(t,Y^2_t,\E\left[Y^2_t\big|\mathcal{G}_t\right]\right)|\mathcal{G}_t\right], \ t\in[0,T].
    \end{align*}

    \noindent (ii) To derive the comparison property for the conditional expectation of the solution, both drivers can be decreasing with respect to $y'$. However, if the drivers are decreasing in $y'$, the pointwise comparison property may not hold (see Example 3.2 in \cite{BLP}). 
\end{remark}
	
\section{Construction via penalization for conditional reflected BSDEs}


Throughout this section, we assume that the dimension for $Y$-term is $1$ and the sub-filtration $\mathbb{G}$ satisfies the usual conditions\footnote{The reason why we add the assumption that $\mathbb{G}$ satisfies the usual condition is that the proof of this section needs the Doob's maximal inequality for $\mathbb{G}$-martingales.}.  We provide an alternative approach to construct the solution to the conditional reflected BSDEs \eqref{nonlinearyz} by the penalization method using the conditional expectation BSDEs. More precisely, let us consider the following family of conditional expectation BSDEs
\begin{equation*}\label{panelization}
Y_t^n=\xi+\int_t^T f(s,Y_s^n,Z_s^n)ds+\int_t^T n(\E[Y_s^n-S_s|\mathcal{G}_s])^-ds-\int_t^T Z_s^n dB_s.
\end{equation*}
By Theorem \ref{wellposedness}, if $f$ satisfies the usual conditions (see Remark \ref{r1}), the above equation admits a unique pair of solution $(Y^n,Z^n)\in \mathcal{S}^2\times \mathcal{H}^2$. Set $K^{n}_t=\int_0^t n(\E[Y_s^n-S_s|\mathcal{G}_s])^- ds$.  We show that $(Y^n,Z^n,K^n)$ converges to $(Y,Z,K)$, which is the solution to the BSDE with conditional reflection. The proof will be divided into the following three lemmas. In the sequel, $C$ will always represent a constant independent of $n$, which may vary from line to line. 

\begin{lemma}\label{estimateYnZn}
Given $\xi\in L^2(\mathcal{F}_T)$ and $S\in\mathcal{S}^2$, suppose that $f$ satisfies the usual assumptions. There exists a constant $C$ independent of $n$, such that
\begin{align*}
&\E\left[\sup_{t\in[0,T]}|Y_t^n|^2\right]+\E\left[\int_0^T |Z_s^n|^2 ds\right]+\E\left[|K^n_T|^2\right]\\
\leq& C\left(\E[\xi^2]+\E\left[\int_0^T |f(t,0,0)|^2 dt\right]+\E\left[\sup_{t\in[0,T]}|S_t|^2\right]\right).
\end{align*}
\end{lemma}

\begin{proof}
Applying It\^{o}'s formula to $e^{\beta t}|{Y}^n_t|^2$, where $\beta$ is a positive constant to be determined later, we have
\begin{equation}\label{eq1.61}\begin{split}
&e^{\beta t}|{Y}^n_t|^2+\int_t^T \beta e^{\beta s}|{Y}^n_s|^2 ds+\int_t^T e^{\beta s}|Z_s^n|^2 ds\\
=&e^{\beta T}|\xi|^2+\int_t^T 2e^{\beta s}{Y}^n_sf(s,Y_s^n,Z_s^n)ds-\int_t^T 2e^{\beta s}{Y}_s^n Z_s^ndB_s\\
&+\int_t^T 2ne^{\beta s}{Y}_s^n(\E[Y_s^n-S_s|\mathcal{G}_s])^-ds.
\end{split}\end{equation}
It is easy to check that 
\begin{equation}\label{eq1.62}\begin{split}
2{Y}^n_sf(s,Y_s^n,Z_s^n)
\leq |f(s,0,0)|^2+\frac{1}{2}|Z^n_s|^2+(1+2\lambda+2\lambda^2)|{Y}^n_s|^2
\end{split}\end{equation}
and 
\begin{equation}\label{eq1.62'}
    \begin{split}
        &n\E\left[\int_t^T e^{\beta s}{Y}_s^n(\E[Y_s^n-S_s|\mathcal{G}_s])^-ds\Big|\mathcal{G}_t\right]=n\int_t^T\E\left[e^{\beta s}{Y}_s^n(\E[Y_s^n-S_s|\mathcal{G}_s])^-|\mathcal{G}_t\right]ds\\
        =&n\int_t^T\E\left[e^{\beta s}({Y}_s^n-S_s)(\E[Y_s^n-S_s|\mathcal{G}_s])^-|\mathcal{G}_t\right]ds+n\int_t^T\E\left[e^{\beta s}{S}_s(\E[Y_s^n-S_s|\mathcal{G}_s])^-|\mathcal{G}_t\right]ds\\
        =&n\int_t^T\E\left[e^{\beta s}\E[{Y}_s^n-S_s|\mathcal{G}_s](\E[Y_s^n-S_s|\mathcal{G}_s])^-|\mathcal{G}_t\right]ds+n\E\left[\int_t^Te^{\beta s}{S}_s(\E[Y_s^n-S_s|\mathcal{G}_s])^-ds\Big|\mathcal{G}_t\right]\\
        \leq &\E\left[\int_t^Te^{\beta s}{S}_sdK^n_s\Big|\mathcal{G}_t\right].
    \end{split}
\end{equation}
Set $\beta=1+2\lambda+2\lambda^2$. Plugging Eq. \eqref{eq1.62} to Eq. \eqref{eq1.61}, taking conditional expectations w.r.t. $\mathcal{G}_t$ on both sides and using \eqref{eq1.62'}, we have
\begin{align*}
\E\left[|Y^n_t|^2+\int_t^T |Z_s^n|^2 ds\Big|\mathcal{G}_t\right]\leq C\E\left[|\xi|^2+\int_0^T |f(s,0,0)|^2ds+\sup_{t\in[0,T]}|S_t||K^n_T|\Big|\mathcal{G}_t\right].
\end{align*}
Taking expectations on both sides yields that 
\begin{equation}
\label{eq1.62''}
\begin{split}
&\sup_{t\in[0,T]}\E\left[|Y^n_t|^2\right]+\E\left[\int_0^T |Z_s^n|^2 ds\right]\\
\leq& C\E\left[|\xi|^2+\int_0^T |f(s,0,0)|^2ds\right]+C\left(\E\left[\sup_{t\in[0,T]}|S_t|^2\right]\right)^{1/2}\left(\E\left[|K^n_T|^2\right]\right)^{1/2}\\
\leq &C\E\left[|\xi|^2+\int_0^T |f(s,0,0)|^2ds+\sup_{t\in[0,T]}|S_t|^2\right]+\varepsilon\E\left[|K^n_T|^2\right].
\end{split}
\end{equation}
On the other hand, note that 
\begin{align*}
    K^n_T=Y^n_0-\xi-\int_0^T f(s,Y^n_s,Z^n_s)ds+\int_0^T Z^n_s dB_s.
\end{align*}
Then, we have 
\begin{equation}\label{eq1.62'''}
    \E[|K^n_T|^2]\leq C\E\left[|\xi|^2+\int_0^T|f(s,0,0)|^2ds+\int_0^T|Z^n_s|^2ds\right]+C\sup_{t\in[0,T]}\E\left[|Y^n_t|^2\right].
\end{equation}
Plugging the above inequality into \eqref{eq1.62''} and choosing $\varepsilon$ sufficiently small, we obtain that
\begin{align*}
    \sup_{t\in[0,T]}\E\left[|Y^n_t|^2\right]+\E\left[\int_0^T |Z_s^n|^2 ds\right]\leq C\E\left[|\xi|^2+\int_0^T|f(s,0,0)|^2ds+\sup_{t\in[0,T]}|S_t|^2\right],
\end{align*}
which together with \eqref{eq1.62'''} implies that 
\begin{align*}
    \sup_{t\in[0,T]}\E\left[|Y^n_t|^2\right]+\E\left[\int_0^T |Z_s^n|^2 ds\right]+\E[|K^n_T|^2]\leq C\E\left[|\xi|^2+\int_0^T|f(s,0,0)|^2ds+\sup_{t\in[0,T]}|S_t|^2\right].
\end{align*}
The desired result follows form the B-D-G inequality.
\end{proof}

In the following, we show that the running supremum of the negative part of the conditional expectation $\E[Y^n_t-S_t|\mathcal{G}_t]$ converges to $0$ under the norm $\|\cdot\|_{\mathcal{S}^2}$, which plays an important role as Lemma 6.1 in \cite{KKPPQ} for classical reflected BSDEs,  Eqs. (6.14)-(6.15) in \cite{CK} for doubly reflected BSDEs and Eq. (29) in \cite{BEH} for mean reflected BSDEs. Moreover, we may obtain the convergence rate and the necessity of the explicit rate is given in Remark \ref{r2} below. Due to the lack of the comparison theorem for general conditional expectation BSDEs (i.e., for the case that the driver is not nondecreasing w.r.t. $y'$), the method used in \cite{KKPPQ} is not valid. In order to obtain the desired result, motivated by \cite{CK} (see Theorem 6.5 and Eq. (6.4)), we propose the following additional assumptions on the driver $f$ and the obstacle $S$.

\begin{itemize}
    \item[($H'_f$)] For any $(\omega,s,z)\in \Omega\times[0,T]\times\mathbb{R}$, there exists a constant $L$ such that $f(\omega,s,0,z)\geq L$ or $f$ is independent of $z$ and 
    \begin{align*}
        \E\left[\sup_{t\in[0,T]}|f(t,0)|^2\right]<\infty.
    \end{align*}
    \item[($H'_S$)] $S$ is an It\^{o} process with representation
    \begin{align}\label{itoprocess}
        S_t=S_0+\int_0^t b_s ds+\int_0^t \sigma_s dB_s,
    \end{align}
    where $b\in\mathcal{S}^2$ and $\sigma\in\mathcal{H}^2$.
\end{itemize}

\begin{lemma}\label{estimateYn-S}
Suppose that $f$ satisfies the usual conditions and $S$ has representation \eqref{itoprocess} with $b,\sigma\in\mathcal{H}^2$. Given $\xi\in L^2(\mathcal{F}_T)$ with $\E[\xi-S_T|\mathcal{G}_T]\geq 0$, we have 
\begin{align}\label{lem3.2 eq1}
    \lim_{n\rightarrow\infty}\E\left[\sup_{t\in[0,T]}|(\E[Y_t^n-S_t|\mathcal{G}_t])^-|^2 \right]=0.
\end{align}
Moreover, suppose that ($H'_f$) and ($H'_S$) hold. Then, there exists a constant $C$ independent of $n$, such that
\begin{align}\label{lem3.2 eq2}
\E\left[\sup_{t\in[0,T]}|(\E[Y_t^n-S_t|\mathcal{G}_t])^-|^2 \right]\leq \frac{C}{n^2}.
\end{align}
\end{lemma}

\begin{proof}
\textbf{Step 1.} We first show that \eqref{lem3.2 eq1} holds.  Set $\widetilde{Y}_t^n=Y^n_t-S_t$. Applying It\^{o}'s formula to $e^{-nt}\widetilde{Y}^n_t$, 
we obtain that 
\begin{equation}\begin{split}\label{tildey}
\widetilde{Y}^n_t=&(\xi-S_T)e^{n(t-T)}+\int_t^T n e^{n(t-s)}\left[\widetilde{Y}^n_s+\left(\E\left[\widetilde{Y}^n_s\big|\mathcal{G}_s\right]\right)^-\right]ds\\
&+\int_t^T e^{n(t-s)}(f(s,Y^n_s,Z^n_s)+b_s)ds-\int_t^T\widetilde{Z}^n_sdB_s\\
\geq &(\xi-S_T)e^{n(t-T)}+\int_t^T n e^{n(t-s)}\left[\widetilde{Y}^n_s+\left(\E\left[\widetilde{Y}^n_s\big|\mathcal{G}_s\right]\right)^-\right]ds\\
&-\int_t^T e^{n(t-s)}(|b_s|+\lambda|Y^n_s|+\lambda|Z^n_s|+|f(s,0,0)|)ds-\int_t^T\widetilde{Z}^n_sdB_s,
\end{split}\end{equation} 
where $\widetilde{Z}^n_t=Z^n_t-\sigma_t$. Noting that $\E[\xi-S_T|\mathcal{G}_t]=\E[\E[\xi-S_T|\mathcal{G}_T]|\mathcal{G}_t]\geq 0$ and taking conditional expectation w.r.t. $\mathcal{G}_t$ on both sides imply that 
\begin{equation*}\begin{split}
\E\left[\widetilde{Y}^n_t\big|\mathcal{G}_t\right]
\geq &\E\left[\int_t^T n e^{n(t-s)}\left[\widetilde{Y}^n_s+\left(\E\left[\widetilde{Y}^n_s\big|\mathcal{G}_s\right]\right)^-\right]ds\Big|\mathcal{G}_t\right]\\
&-\E\left[\int_t^T e^{n(t-s)}(|b_s|+\lambda|Y^n_s|+\lambda|Z^n_s|+|f(s,0,0)|)ds\Big|\mathcal{G}_t\right].
\end{split}\end{equation*} 
Observe that 
\begin{align*}
 &\E\left[\int_t^T n e^{n(t-s)}\left[\widetilde{Y}^n_s+\left(\E\left[\widetilde{Y}^n_s\big|\mathcal{G}_s\right]\right)^-\right]ds\Big|\mathcal{G}_t\right]\\   
 =&\int_t^T n e^{n(t-s)}\E\left[\widetilde{Y}^n_s+\left(\E\left[\widetilde{Y}^n_s\big|\mathcal{G}_s\right]\right)^-\Big|\mathcal{G}_t\right]ds\\
 =&\int_t^T n e^{n(t-s)}\E\left[\E\left[\widetilde{Y}^n_s\big|\mathcal{G}_s\right]+\left(\E\left[\widetilde{Y}^n_s\big|\mathcal{G}_s\right]\right)^-\Big|\mathcal{G}_t\right]ds\geq 0
\end{align*}
All the above analysis implies that 
\begin{align*}
    \E\left[\widetilde{Y}^n_t\big|\mathcal{G}_t\right]
\geq -\E\left[\int_t^T e^{n(t-s)}(|b_s|+\lambda|Y^n_s|+\lambda|Z^n_s|+|f(s,0,0)|)ds\Big|\mathcal{G}_t\right].
\end{align*}
Consequently, we have
\begin{align*}
    \left(\E\left[\widetilde{Y}^n_t\big|\mathcal{G}_t\right]\right)^-
&\leq 
\E\left[\int_t^T e^{n(t-s)}(|b_s|+\lambda|Y^n_s|+\lambda|Z^n_s|+|f(s,0,0)|)ds\Big|\mathcal{G}_t\right].\\
&\leq \left(\int_t^T e^{2n(t-s)}ds\right)^{1/2}\left(\int_t^T \left(\E[|b_s|+\lambda|Y^n_s|+\lambda|Z^n_s|+|f(s,0,0)||\mathcal{G}_t]\right)^2dt\right)^{1/2}\\
&\leq \frac{C}{n^{1/2}} \left(\int_0^T \left(\E[|b_s|+|Y^n_s|+|Z^n_s|+|f(s,0,0)||\mathcal{G}_t]\right)^2 dt\right)^{1/2}.
\end{align*}
Applying Lemma \ref{estimateYnZn} and Doob's maximal inequality yields the desired result.

\textbf{Step 2.} Now, we show \eqref{lem3.2 eq2} holds. We only prove the case that $f$ depends on $z$ and is bounded from below. In this case, noting that 
\begin{align*}
    f(s,Y^n_s,Z^n_s)\geq f(s,0,Z^n_s)-\lambda|Y^n_s|\geq L-\lambda|Y^n_s|,
\end{align*}
thus, \eqref{tildey} turns into 
\begin{align*}
\widetilde{Y}^n_t
\geq &(\xi-S_T)e^{n(t-T)}+\int_t^T n e^{n(t-s)}\left[\widetilde{Y}^n_s+\left(\E\left[\widetilde{Y}^n_s\big|\mathcal{G}_s\right]\right)^-\right]ds\\
&-\int_t^T e^{n(t-s)}(|b_s|+\lambda|Y^n_s|)ds+\frac{L}{n}-\int_t^T\widetilde{Z}^n_sdB_s.
\end{align*} 
By a similar analysis as in Step 1, we obtain that 
\begin{align*}
    \E\left[\widetilde{Y}^n_t\big|\mathcal{G}_t\right]
\geq \frac{L}{n}
-\E\left[\int_t^T e^{n(t-s)}(|b_s|+\lambda|Y^n_s|)ds\Big|\mathcal{G}_t\right].
\end{align*}
Consequently, we have
\begin{align*}
    \left(\E\left[\widetilde{Y}^n_t\big|\mathcal{G}_t\right]\right)^-
&\leq \frac{L^-}{n}
+\E\left[\int_t^T e^{n(t-s)}(|b_s|+\lambda|Y^n_s|)ds\Big|\mathcal{G}_t\right]\\
&\leq \frac{L^-}{n}
+\frac{1}{n}\E\left[\sup_{s\in[0,T]}(|b_s|+\lambda|Y^n_s|)\Big|\mathcal{G}_t\right].
\end{align*}
Applying Lemma \ref{estimateYnZn} and Doob's maximal inequality again, we obtain the desired result.
\end{proof}

The following lemma indicates that the sequences $\{Y^n\}_{n\in\mathbb{N}}$, $\{Z^n\}_{n\in\mathbb{N}}$ and $\{K^n\}_{n\in\mathbb{N}}$ are Cauchy sequences in $\mathcal{S}^2$, $\mathcal{H}^2$ and $\mathcal{S}^2$, respectively.

\begin{lemma}\label{limit}
    Suppose that $f$ satisfies the usual conditions and $(H'_f)$-$(H'_S)$ hold. Given $\xi\in L^2(\mathcal{F}_T)$ with $\E[\xi-S_T|\mathcal{G}_T]\geq0$,  we have 
    \begin{align*}
        &\lim_{m,n\rightarrow \infty}\E\left[\sup_{t\in[0,T]}|Y^m_t-Y^n_t|^2\right]=0,\\ &\lim_{m,n\rightarrow \infty}\E\left[\int_0^T|Z^m_t-Z^n_t|^2dt\right]=0, \\ &\lim_{m,n\rightarrow \infty}\E\left[\sup_{t\in[0,T]}|K^m_t-K^n_t|^2\right]=0.
    \end{align*}
\end{lemma}

\begin{proof}
    For simplicity, we write $A^{m,n}_t=A^m_t-A^n_t$ for $A=Y,Z,K$ and $f^{m,n}_t=f(t,Y^m_t,Z^m_t)-f(t,Y^n_t,Z^n_t)$. Applying It\^{o}'s formula to $e^{\beta t}|Y^{m,n}_t|^2$, we have
    \begin{equation}\label{eq1.64}\begin{split}
        &e^{\beta t}|{Y}^{m,n}_t|^2+\int_t^T \beta e^{\beta s}|{Y}^{m,n}_s|^2 ds+\int_t^T e^{\beta s}|Z_s^{m,n}|^2 ds\\
=&\int_t^T 2e^{\beta s}{Y}^{m,n}_sf^{m,n}_sds-\int_t^T 2e^{\beta s}{Y}_s^{m,n} Z_s^{m,n}dB_s+\int_t^T 2e^{\beta s}{Y}_s^{m,n}dK^{m,n}_s\\
\leq &2(\lambda+\lambda^2)\int_t^T e^{\beta s}|Y^{m,n}_s|^2ds+\frac{1}{2}\int_t^T e^{\beta s}|Z^{m,n}|^2 ds \\
&-\int_t^T 2e^{\beta s}{Y}_s^{m,n} Z_s^{m,n}dB_s+\int_t^T 2e^{\beta s}{Y}_s^{m,n}dK^{m,n}_s.
    \end{split}\end{equation}
Simple calculation yields that for any $u\in[0,t]$, 
\begin{equation}\begin{split}\label{ymnkmn}
    &\E\left[\int_t^T e^{\beta s}{Y}_s^{m,n}dK^{m,n}_s\Big|\mathcal{G}_u\right]\\
    =&\int_t^T e^{\beta s}\E\left[\left(\widetilde{Y}^m_s-\widetilde{Y}^n_s\right)\left(m\left(\E\left[\widetilde{Y}^m_s\big|\mathcal{G}_s\right]\right)^--n\left(\E\left[\widetilde{Y}^n_s\big|\mathcal{G}_s\right]\right)^-\right)\Big|\mathcal{G}_u\right]ds\\
    =&\int_t^T e^{\beta s}\E\left[\left(\E\left[\widetilde{Y}^m_s\big|\mathcal{G}_s\right]-\E\left[\widetilde{Y}^n_s\big|\mathcal{G}_s\right]\right)\left(m\left(\E\left[\widetilde{Y}^m_s\big|\mathcal{G}_s\right]\right)^--n\left(\E\left[\widetilde{Y}^n_s\big|\mathcal{G}_s\right]\right)^-\right)\Big|\mathcal{G}_u\right]ds\\
    \leq &(m+n)\E\left[\int_t^T e^{\beta s}\left(\E\left[\widetilde{Y}^m_s\big|\mathcal{G}_s\right]\right)^-\left(\E\left[\widetilde{Y}^n_s\big|\mathcal{G}_s\right]\right)^- \Big|\mathcal{G}_u\right]\\
    \leq &e^{\beta T}\E\left[\sup_{s\in[0,T]} \left(\E\left[\widetilde{Y}^n_s\big|\mathcal{G}_s\right]\right)^-K^m_T+\sup_{s\in[0,T]} \left(\E\left[\widetilde{Y}^m_s\big|\mathcal{G}_s\right]\right)^-K^n_T\Big|\mathcal{G}_u\right],
    \end{split}
\end{equation}
where $\widetilde{Y}_t^i=Y^i_t-S_t$ for $i=m,n$. Choosing $\beta=2(\lambda+\lambda^2)$, taking expectations on both sides of \eqref{eq1.64}, we obtain that for any $t\in[0,T]$, 
\begin{align*}
    &\E\left[|Y^{m,n}_t|^2\right]+\E\left[\int_t^T |Z^{m,n}_s|^2 ds\right]\\\leq &C\E\left[\sup_{s\in[0,T]} \left(\E\left[\widetilde{Y}^n_s\big|\mathcal{G}_s\right]\right)^-K^m_T+\sup_{s\in[0,T]} \left(\E\left[\widetilde{Y}^m_s\big|\mathcal{G}_s\right]\right)^-K^n_T\right]\\
    \leq &C\left(\E\left[\sup_{s\in[0,T]} \Big|\left(\E\left[\widetilde{Y}^n_s\big|\mathcal{G}_s\right]\right)^-\Big|^2\right]\right)^{1/2}\left(\E\left[|K^m_T|^2\right]\right)^{1/2}\\ &+C\left(\E\left[\sup_{s\in[0,T]} \Big|\left(\E\left[\widetilde{Y}^m_s\big|\mathcal{G}_s\right]\right)^-\Big|^2\right]\right)^{1/2}\left(\E\left[|K^n_T|^2\right]\right)^{1/2}.
\end{align*}
By Lemma \ref{estimateYnZn} and Lemma \ref{estimateYn-S}, we have 
\begin{equation}\label{eq1.65}
    \lim_{m,n\rightarrow \infty} \sup_{t\in[0,T]}\E\left[|Y^{m,n}_t|^2\right]=0, \ \lim_{m,n\rightarrow\infty}\E\left[\int_0^T |Z^{m,n}_s|^2 ds\right]=0.
\end{equation}
Recalling \eqref{eq1.64} and $\beta=2(\lambda+\lambda^2)$, we have 
\begin{align*}
    \sup_{t\in[0,T]}|Y^{m,n}_t|^2\leq &C\Bigg(\sup_{t\in[0,T]}\Big|\int_t^T {Y}_s^{m,n} Z_s^{m,n}dB_s\Big|+\int_0^T|Y^{m,n}_s|d(K^m_s+K^n_s)\Bigg)\\
    \leq &C\Bigg(\int_0^T |{Y}_s^{m,n}|ds \left[\sup_{s\in[0,T]}\left(m\left(\E\left[\widetilde{Y}^m_s\big|\mathcal{G}_s\right]\right)^-+n\left(\E\left[\widetilde{Y}^n_s\big|\mathcal{G}_s\right]\right)^-\right)\right]\\
    &\ \ \ \ \ \ +\sup_{t\in[0,T]}\Big|\int_t^T {Y}_s^{m,n} Z_s^{m,n}dB_s\Big|\Bigg).
\end{align*}
Applying the B-D-G inequality, we obtain that 
\begin{equation*}\begin{split}\label{eq1.65'}
    \E\left[\sup_{t\in[0,T]}\Big|\int_t^T {Y}_s^{m,n} Z_s^{m,n}dB_s\Big|\right]\leq& C\E\left[\left(\int_0^T |Y^{m,n}_s Z^{m,n}_s|^2ds\right)^{1/2}\right]\\
    \leq &C\E\left[\sup_{t\in[0,T]}|Y^{m,n}_t|\left(\int_0^T | Z^{m,n}_s|^2ds\right)^{1/2}\right]\\
    \leq& C\left(\E\left[\sup_{t\in[0,T]}|Y^{m,n}_t|^2\right]\right)^{1/2}\left(\E\left[\int_0^T | Z^{m,n}_s|^2ds\right]\right)^{1/2}\\
    \leq &\varepsilon \E\left[\sup_{t\in[0,T]}|Y^{m,n}_t|^2\right]+C\E\left[\int_0^T | Z^{m,n}_s|^2ds\right].
    \end{split}
\end{equation*}
By Lemma \ref{estimateYn-S}, we have
\begin{align*}
    &\E\left[\int_0^T |{Y}_s^{m,n}|ds \sup_{s\in[0,T]}m\left(\E\left[\widetilde{Y}^m_s\big|\mathcal{G}_s\right]\right)^-\right]\\
    \leq &\left(\E\left[\left(\int_0^T |{Y}_s^{m,n}|ds\right)^2\right]\right)^{1/2}\left(m^2\E\left[\sup_{s\in[0,T]}\Bigg|\left(\E\left[\widetilde{Y}^m_s\big|\mathcal{G}_s\right]\right)^-\Bigg|^2\right]\right)^{1/2}\\
    \leq & C\left(\E\left[\int_0^T |{Y}_s^{m,n}|^2ds\right]\right)^{1/2}.
\end{align*}
Choosing $\varepsilon<1$, all the above analysis indicates that 
\begin{align*}
    \E\left[\sup_{t\in[0,T]}|Y^{m,n}_t|^2\right]\leq C\left(\left(\E\left[\int_0^T |{Y}_s^{m,n}|^2ds\right]\right)^{1/2}+\E\left[\int_0^T | Z^{m,n}_s|^2ds\right]\right).
\end{align*}
It follows from \eqref{eq1.65} that  
\begin{align*}
    \lim_{m,n\rightarrow \infty}\E\left[\sup_{t\in[0,T]}|Y^m_t-Y^n_t|^2\right]=0.
\end{align*}
Finally, note that 
\begin{align*}
    K^{m,n}_t=Y^{m,n}_0-Y^{m,n}_t+\int_0^t Z^{m,n}_sdB_s-\int_0^t f^{m,n}_sds.
\end{align*}
Simple calculation yields that 
\begin{align*}
    \lim_{m,n\rightarrow \infty}\E\left[\sup_{t\in[0,T]}|K^m_t-K^n_t|^2\right]=0.
\end{align*}
The proof is complete.
\end{proof}

\begin{remark}\label{r2}
    Consider the case that $\mathbb{G}=\mathbb{F}$. In this case, we have 
    $K^n_t=\int_0^t n(Y^n_s-S_s)^-ds$. Simple calculation yields that 
    \begin{equation}\label{eq1.65''}\begin{split}
        \int_t^T Y^{m,n}_s dK^{m,n}_s=&\int_t^T m(Y^m_s-S_s)(Y^m_s-S_s)^-ds-\int_t^T(Y^m_s-S_s)dK^n_s\\
        &+\int_t^T n(Y^n_s-S_s)(Y^n_s-S_s)^-ds-\int_t^T(Y^n_s-S_s)dK^m_s\\
        \leq &\int_t^T(Y^m_s-S_s)^-dK^n_s+\int_t^T(Y^n_s-S_s)^-dK^m_s\\
        \leq &\sup_{s\in[0,T]}(Y^m_s-S_s)^- K^n_T+\sup_{s\in[0,T]}(Y^n_s-S_s)^- K^m_T.
        \end{split}
    \end{equation}
    Hence, provided that 
    \begin{align*}
        \lim_{n\rightarrow\infty} \E\left[\sup_{t\in[0,T]}|(Y_t^n-S_t)^-|^2 \right]=0
    \end{align*}
    and recalling the proof in Section 6 of \cite{KKPPQ}, we obtain that $\{Y^n\}_{n\in\mathbb{N}}$  is a Cauchy sequence in $\mathcal{S}^2$. However, for the general case that $\mathbb{G}\neq \mathbb{F}$, the following inequality similar to \eqref{eq1.65''} may not hold
    \begin{align*}
       \int_t^T Y^{m,n}_s dK^{m,n}_s\leq  \sup_{s\in[0,T]}(\E[Y^m_s-S_s|\mathcal{G}_s])^- K^n_T+\sup_{s\in[0,T]}(\E[Y^n_s-S_s|\mathcal{G}_s])^- K^m_T.
    \end{align*}
    Actually, we only have Eq. \eqref{ymnkmn}. Therefore, in order to derive that
    \begin{align*}
        \lim_{m,n\rightarrow \infty}\E\left[\sup_{t\in[0,T]}|Y^m_t-Y^n_t|^2\right]=0,
    \end{align*}
    we need to establish an explicit convergence rate as stated in Lemma \ref{estimateYn-S}.
\end{remark}


Now, we state the main result in this section.
\begin{theorem}\label{main2}
Suppose that $f$ satisfies the usual conditions and $(H'_f)$-$(H'_S)$ hold. Given $\xi\in L^2(\mathcal{F}_T)$ with $\E[\xi-S_T|\mathcal{G}_T]\geq 0$, the BSDE with conditional reflection \eqref{nonlinearyz} has a unique solution $(Y,Z,K)$. Furthermore, $(Y,Z,K)$ is the limit of $(Y^n,Z^n,K^n)$.
\end{theorem}

\begin{proof}
Uniqueness is a direct consequence of Theorem 2.3 in \cite{HHL}. By Lemma \ref{limit}, there exists a triple of processes $(Y,Z,K)$, such that 
 \begin{align*}
        &\lim_{n\rightarrow \infty}\E\left[\sup_{t\in[0,T]}|Y_t-Y^n_t|^2\right]=0,\\ &\lim_{n\rightarrow \infty}\E\left[\int_0^T|Z_t-Z^n_t|^2dt\right]=0, \\ &\lim_{n\rightarrow \infty}\E\left[\sup_{t\in[0,T]}|K_t-K^n_t|^2\right]=0.
    \end{align*}
It remains to prove that the triple of limit processes $(Y,Z,K)$ is the solution to the BSDE with conditional reflection. First, applying Lemma \ref{estimateYn-S}, we obtain that    
\begin{align*}
\E\left[\sup_{t\in[0,T]}|(\E[Y_t-S_t|\mathcal{G}_t])^-|^2 \right]=0,
\end{align*}
which implies that $\E[Y_t|\mathcal{G}_t]\geq \E[S_t|\mathcal{G}_t]$, for any $t\in[0,T]$. The proof for the Skorokhod condition is similar with the one for the classical reflected case (see Section 6 in \cite{KKPPQ}). So we omit it. The proof is complete. 
\end{proof}

\begin{remark}\label{remark3.6}
Recall that Theorem 2.7 in \cite{HHL} also gives the well-posedness of conditional reflected BSDE \eqref{nonlinearyz}. Although the driver $f$ is only required to satisfy the usual conditions and the obstacle $S$ belongs to $\mathcal{S}^2$ in \cite{HHL}, the subfiltration $\mathbb{G}$ should satisfy the following  assumptions
     \begin{itemize}
        \item[(i)] $\mathbb{G}$ satisfies the usual conditions of right-continuity and completeness;
        \item [(ii)] $\mathbb{G}$ is left-quasi-continuous. That is, $\mathbb{G}$ is left-continuous along stopping times.
    \end{itemize}
Compared with the result in \cite{HHL}, our assumption for $\mathbb{G}$ is weaker.
\end{remark}


\begin{thebibliography}{99}

\bibitem{BEH} Briand, P., Elie, R. and Hu, Y. (2018) BSDEs with mean reflection. The Annals of Applied Probability, 28: 482-510.

\bibitem{BH1} Briand, P. and Hu, Y. (2006) BSDE with quadratic growth and unbounded terminal value. Probability Theory and Related Fields, 136: 604–618.


\bibitem{BH2} Briand, P. and Hu, Y. (2008) Quadratic BSDEs with convex generators and unbounded terminal conditions. Probability Theory and Related Fields, 141: 543-567.

\bibitem{BDL} Buckdahn, R., Djehiche, B. and Li, J. (2011) A general stochastic maximum principle for SDEs of mean-field type. Appl. Math. Optim. 64(2): 197–216.

\bibitem{BDLP} Buckdahn, R., Djehiche, B., Li, J. and Peng, S. (2009) Mean-field backward stochastic differential equations: a limit approach. Ann. Probab., 37(4): 1524-1565.

\bibitem{BLP} Buckdahn, R., Li, J. and Peng, S. (2009) Mean-field backward stochastic differential equations and related partial differential equations. Stochastic Processes and their Applications, 119: 3133-3154.

\bibitem{BLPR} Buckdahn, R., Li, J., Peng, S. and Rainer, C. (2017) Mean-field stochastic differential equations and associated PDEs. Ann. Probab., 45: 824–878.




\bibitem{CK} Cvitanic, J. and Karatzas, I. (1996) Backward stochastic differential equations with reflection and Dynkin games. Ann. Probab.,  24(4): 2024-2056.

\bibitem{DEH} Djehiche, B., Elie, R. and Hamad\`{e}ne, S. (2023) Mean-field reflected backward stochastic differential equations. Ann. Probab., 33(4): 2493-2518.

\bibitem{DE} Duffie, E. and Epstein, L. (1992) Stochastic differential utility. Econometrica, 60: 353-394.



\bibitem{KKPPQ} El Karoui, N., Kapoudjian, C., Pardoux, E., Peng, S and Quenez,  M.C. (1997) Reflected solutions of backward SDE's, and related obstacle problems for PDE's. The Annals of Probability,  23(2): 702-737.


\bibitem{EPQ} El Karoui, N., Peng, S. and Quenez M.C. (1997) Backward stochastic differential equations in finance. Math. Finance, 7: 1-71.


\bibitem{HL} Hamad\`{e}ne, S. and  Lepeltier, J.-P. (2000) Reflected BSDE's and mixed game problem. Stochastic Process. Appl.,  85: 177-188.


\bibitem{HHL} Hu, Y., Huang, J. and Li, W. (2024) Backward stochastic differential equations with conditional reflection and related recursive optimal control problems. SIAM J. Control Optim., 62(5): 2557-2589.

\bibitem{HT} Hu, Y. and Tang, S. (2010). Multi-dimensional BSDE with oblique reflection and optimal switching. Probab. Theory Related Fields 147: 89–121.

\bibitem{KP} Kaden, S. and Potthoff, J. (2004) Progressive stochastic processes and an application to the It\^{o} integral. Stoch. Anal. Appl. 22(4): 843–865.

\bibitem{Kobylanski} Kobylanski, M. (2000) Backward stochastic diﬀerential equations and partial diﬀerential equations
with quadratic growth. Ann. Probab., 28: 558–602.

\bibitem{LS} Lepeltier, J.-P. and San Martin, J. (1998) Existence for BSDE with superlinear-quadratic coefficient. Stoch. Stoch. Rep., 63(3-4): 227–240.

\bibitem{Li} Li, H. (2024) Backward stochastic differential equations with double mean reflections. Stochastic Processes and their Applications, 173: 104371.

\bibitem{L1} Li, J. (2012) Stochastic maximum principle in the mean-field controls. Automatica, 48(2): 366–373.

\bibitem{L2} Li, J. (2014) Reflected mean-field backward stochastic differential equations: approximation and associated nonlinear PDEs. Journal of Mathematical Analysis and Applications, 413: 47–68.

\bibitem{LXP} Li, J., Xing, C. and Peng, Y. (2021) Comparison theorem for multi-dimensional general mean-field BDSDEs. Acta Mathematica Scientia, 41B(2): 535-551.

\bibitem{LYY} Li, J., Yu, Z. and Yue, W. (2025) Indefinite linear-quadratic optimal control problems of backward stochastic differential equations with partial information, arXiv: 2507.14992v1.










\bibitem{PP} Pardoux, E. and Peng, S. (1990) Adapted solutions of backward equations. Systerm and Control Letters, 14: 55-61.

\bibitem{PP92} Pardoux, E. and Peng, S. (1992) Backward stochastic differential equations and quasilinear parabolic partial differential equations. Stochastic Partial Differential Equations and their Applications, Proc. IFIP, LNCIS 176: 200-217.


\bibitem{PX} Peng, S. and Xu, M. (2010) Reflected BSDE with a constraint and its applications in an incomplete market. Bernoulli, 16: 614–640.

\bibitem{WWY} Wang, G., Wang, W. and Yan, Z. (2021) Linear quadratic control of backward stochastic differential
equation with partial information. Appl. Math. Comput., 403: 126164. 

\end{thebibliography}
 \end{document}